\documentclass[11pt,a4paper]{article}
\usepackage[margin=1in]{geometry}
\usepackage{amsmath}
\usepackage{amssymb}
\usepackage{amsthm}
\usepackage{setspace}
\usepackage{indentfirst}
\numberwithin{equation}{section}
\theoremstyle{plain}
\newtheorem{theorem}{Theorem}[section]
\newtheorem{lemma}[theorem]{Lemma}
\newtheorem{proposition}[theorem]{Proposition}
\newtheorem{definition}[theorem]{Definition}
\newtheorem{example}[theorem]{Example}
\newtheorem{remark}[theorem]{Remark}

\title{Fixed Points of Asymptotic Pointwise Contractions under Local Uniform Convergence}
\author{Jie Shi \\
Department of Mathematics and Statistics, Hubei Engineering University \\
272 Traffic Avenue, Xiaogan, 432000, Hubei Province, P.R. China}
\date{April 6, 2026}

\begin{document}
\maketitle

\begin{abstract}
We introduce a weak asymptotic version of nonlinear contraction, termed \emph{asymptotic pointwise contraction}. 
For a mapping on a metric space, this notion requires the existence of a sequence of functions that dominate the distances between the $n$-th iterates of any two points. 
The sequence is assumed to converge pointwise to a limit function, and the convergence is required to be uniform on every bounded set (i.e., locally uniform). 
The limit function is then controlled by a Boyd--Wong type condition: there exists a nondecreasing, right upper semicontinuous function strictly below the identity on positive numbers, and the limit function is bounded above by this function evaluated at a maximum term that involves not only the distance between the two points but also distances from each point to its image and mutual distances between each point and the image of the other. 
By standard analytic arguments we prove that if the mapping is continuous on a complete metric space and possesses a bounded orbit, then its iterates converge to a unique fixed point. 
This result extends Kirk's asymptotic contraction theorem by replacing global uniform convergence on $[0,\infty)$ with the weaker condition of local uniform convergence.
\end{abstract}


\section{Introduction}

The first rigorous proof that every continuous self‑map of a compact convex subset of $\mathbb{R}^n$ admits at least one fixed point was given by Brouwer in 1912 \cite{brouwer}. 
Denote by $B^{n}=\{x \in \mathbb{R}^{n}:\|x\| \leq 1\}$ the closed unit ball in $\mathbb{R}^{n}$. 
For any continuous $f: B^{n} \to B^{n}$, there exists $x^{*} \in B^{n}$ such that $f(x^{*})=x^{*}$. 
In other words, every continuous map from a nonempty compact convex subset of $\mathbb{R}^{n}$ into itself has a fixed point. 
This seminal result laid the foundations of modern algebraic topology and fixed‑point theory.

Banach \cite{banach} (1922) formulated the classical contraction mapping principle, which became a pillar of metric fixed point theory. 
On a complete metric space $(X, d)$, a contraction $T$ satisfies
\begin{equation}\label{eq:banach}
d(Tx,Ty)\leq \alpha\, d(x,y)
\end{equation}
for some constant $\alpha\in(0,1)$ and all $x,y\in X$. 
Banach's theorem guarantees that for any starting point $x$, the iterates $\{T^{n} x\}$ converge to a unique fixed point of $T$.

In 1930, Schauder \cite{schauder} extended the fixed‑point theorem to infinite‑dimensional Banach spaces: if $X$ is a Banach space, $C \subseteq X$ a nonempty closed convex set, and $T: C \to C$ a continuous map with $T(C)$ relatively compact, then $T$ possesses a fixed point in $C$.

The study of nonexpansive mappings, i.e.,
\begin{equation}\label{eq:nonexpansive}
d(T x, T y) \leq d(x, y)\qquad(\forall x,y\in X),
\end{equation}
began in the 1960s. 
Such mappings may fail to have fixed points unless additional geometric conditions are imposed on the space.

Rakotch \cite{rakotch} (1962) proposed a nonlinear contractive condition:
\begin{equation}\label{eq:rakotch}
d(T x, T y) \leq \alpha(d(x, y))\, d(x, y)\qquad(\forall x,y\in X),
\end{equation}
where $\alpha$ is decreasing and $\alpha(t)<1$ for $t>0$. 
He proved that such maps have a unique fixed point and the iterates converge to it. 
In the same year, Edelstein \cite{edelstein} showed that a strictly nonexpansive map on a compact metric space admits a unique fixed point.

In 1965, Browder \cite{browder}, G\"ohde, and Kirk \cite{kirk1965} independently established fundamental fixed point results for nonexpansive mappings in uniformly convex or reflexive Banach spaces with normal structure. 
These are collectively known as the Browder–G\"ohde–Kirk theorem.

Petryshyn \cite{petryshyn} (1966) gave a convergence theorem for demicompact nonexpansive mappings.

Boyd and Wong \cite{boydwong} (1969) unified many nonlinear contraction theorems by introducing the condition
\begin{equation}\label{eq:boyd-wong}
d(T x, T y) \leq \psi(d(x, y))\qquad(\forall x,y\in X),
\end{equation}
where $\psi:[0,\infty)\to[0,\infty)$ is nondecreasing, right upper semicontinuous, and $\psi(t)<t$ for all $t>0$.

\'{C}iri\'{c} \cite{ciric} (1974) introduced quasi‑contractions:
\begin{equation}\label{eq:ciric-quasi}
d(Tx,Ty)\leq \alpha \cdot \max\left\{ d(x,y),\, d(x,Tx),\, d(Ty,y),\, d(Tx,y),\, d(Ty,x)\right\},
\end{equation}
with $0<\alpha<1$.

Kirk \cite{kirk2003} (2003) introduced asymptotic contractions: there exist functions $\psi_n:[0,\infty)\to[0,\infty)$ such that
\begin{equation}\label{eq:kirk-asymp}
d\left(T^{n} x, T^{n} y\right) \leq \psi_{n}(d(x, y)),
\end{equation}
and $\psi_n$ converges uniformly on $[0,\infty)$ to a function $\psi$ satisfying the classical Boyd–Wong condition. 
Kirk proved that a continuous mapping having at least one bounded orbit possesses a unique fixed point, and all iterates converge to it.

In 2023, Lindstr{\o}m and Ross \cite{lindstrom} extended Kirk's result using a nonstandard analysis framework.

\medskip\noindent
\textbf{Our improvement.} 
We relax the global uniform convergence in Kirk's theorem to \emph{local uniform convergence} (uniform convergence on every bounded subset). 
Moreover, the contractive condition involves a maximum term
\begin{equation}
M(x,y)=\max\bigl\{d(x,y),\, d(Tx,x),\, d(Ty,y),\, d(Ty,x),\, d(Tx,y)\bigr\}
\end{equation}
instead of only the distance $d(x,y)$. 
This makes our assumption strictly weaker and more general than Kirk's original condition.

\section{Preliminaries}
Throughout this paper, $(X,d)$ denotes a metric space. 
For a mapping $T:X\to X$ and a point $x\in X$, the \emph{orbit} of $x$ is defined as
\begin{equation}
\mathcal{O}(x)=\{T^n x: n\in\mathbb{N}\}.
\end{equation}
A subset $B\subset X$ is \emph{bounded} if its diameter
\begin{equation}
\operatorname{diam}(B)=\sup\{d(u,v):u,v\in B\}
\end{equation}
is finite.

\begin{definition}[Right upper semicontinuity]\label{def:rusc}
A function $\psi:[0,\infty)\to[0,\infty)$ is \emph{right upper semicontinuous} at $t_0\ge0$ if
\begin{equation}
\limsup_{t\to t_0^+} \psi(t) \le \psi(t_0).
\end{equation}
It is called right upper semicontinuous on $[0,\infty)$ if it is right upper semicontinuous at every point $t_0\ge0$.
\end{definition}

\begin{definition}[Classical Boyd--Wong condition]\label{def:bw}
A function $\psi:[0,\infty)\to[0,\infty)$ satisfies the \emph{classical Boyd--Wong condition} if it is nondecreasing, right upper semicontinuous, and $\psi(t)<t$ for all $t>0$.
\end{definition}

\begin{definition}[Asymptotic pointwise contraction]\label{def:main}
Let $T:X\to X$ be a mapping. 
Then $T$ is called an \emph{asymptotic pointwise contraction} if there exists a sequence of functions $\varphi_n: X\times X\to[0,\infty)$ such that:
\begin{enumerate}
    \item For all $x,y\in X$ and all $n\in\mathbb{N}$,
    \begin{equation}
    d(T^n x, T^n y) \le \varphi_n(x,y).
    \end{equation}
    \item The sequence $\{\varphi_n\}$ converges pointwise to a function $\varphi: X\times X\to[0,\infty)$, and the convergence is uniform on $B\times B$ whenever $B\subset X$ is bounded.
    \item The limit function $\varphi$ satisfies a \emph{Boyd--Wong type condition}: there exists a function $\psi:[0,\infty)\to[0,\infty)$ obeying the classical Boyd--Wong condition such that
    \begin{equation}
    \varphi(x,y) \le \psi\bigl(M(x,y)\bigr),
    \end{equation}
    where
    \begin{equation}
    M(x,y)=\max\bigl\{d(x,y),\, d(Tx,x),\, d(Ty,y),\, d(Ty,x),\, d(Tx,y)\bigr\}.
    \end{equation}
\end{enumerate}
\end{definition}

The following auxiliary function plays a key role in converting a possibly non-monotone function into a monotone one.

\begin{definition}\label{def:g}
Let $\psi$ satisfy the classical Boyd--Wong condition. Define
\begin{equation}
g(t) = \sup_{0\le s\le t} \psi(s),\quad t\ge0.
\end{equation}
\end{definition}

\begin{lemma}\label{lem:g-properties}
The function $g$ has the following properties:
\begin{enumerate}
    \item $g$ is nondecreasing, $g(0)=0$, and $g(t)\le t$ for all $t\ge0$.
    \item $g(t)<t$ for every $t>0$.
    \item $g$ is right continuous on $[0,\infty)$.
\end{enumerate}
\end{lemma}

\begin{proof}
(1) Monotonicity and $g(0)=0$ follow directly from the definition. Since $\psi(s)\le s$ for all $s\ge0$, we have $g(t)\le t$.

(2) Suppose for contradiction that $g(t)=t$ for some $t>0$. 
Then there exists a sequence $s_n\in[0,t]$ such that $\psi(s_n)\to t$. 
By compactness of $[0,t]$, we may assume $s_n\to s\in[0,t]$. 
Because $\psi$ is nondecreasing, for any $\varepsilon>0$ there exists $N$ such that $s_n\le s+\varepsilon$ for all $n\ge N$. Hence $\psi(s_n)\le\psi(s+\varepsilon)$. Taking $\limsup$,
\begin{equation}
t = \limsup_{n\to\infty}\psi(s_n) \le \psi(s+\varepsilon).
\end{equation}
Letting $\varepsilon\to0^+$ and using right upper semicontinuity of $\psi$ (which implies $\lim_{\varepsilon\to0^+}\psi(s+\varepsilon)\le\psi(s)$), we obtain $t\le\psi(s)$. 
If $s>0$, then $\psi(s)<s$ by the Boyd--Wong condition, so $t\le\psi(s)<s\le t$, a contradiction. 
If $s=0$, then $\psi(s_n)\to\psi(0)=0$, so $t=0$, contradiction. Hence $g(t)<t$.

(3) Fix $t_0\ge0$ and $\varepsilon>0$. 
By right upper semicontinuity of $\psi$, there exists $\delta>0$ such that
\begin{equation}
\psi(t)\le\psi(t_0)+\varepsilon\quad\text{for all }t\in[t_0,t_0+\delta).
\end{equation}
For any $t\in(t_0,t_0+\delta)$,
\begin{equation}
g(t)\le\max\{g(t_0),\psi(t_0)+\varepsilon\}\le g(t_0)+\varepsilon.
\end{equation}
Since $g$ is nondecreasing, $g(t_0)\le g(t)\le g(t_0)+\varepsilon$. 
Thus $g$ is right continuous at $t_0$.
\end{proof}

\begin{proposition}\label{prop:limsup-g}
Let $\{a_n\}$ be a bounded nonnegative sequence and let $g$ be nondecreasing and right continuous. Then
\begin{equation}
\limsup_{n\to\infty} g(a_n) \le g\bigl(\limsup_{n\to\infty} a_n\bigr).
\end{equation}
\end{proposition}

\begin{proof}
Let $L=\limsup\limits_{n\to\infty} a_n$. For any $\varepsilon>0$, there exists $N\in\mathbb{N}$ such that
\begin{equation}
a_n\le L+\varepsilon\quad\forall n\ge N.
\end{equation}
Since $g$ is nondecreasing,
\begin{equation}
g(a_n)\le g(L+\varepsilon)\quad\forall n\ge N.
\end{equation}
Taking limsup on both sides yields
\begin{equation}
\limsup_{n\to\infty}g(a_n)\le g(L+\varepsilon).
\end{equation}
Letting $\varepsilon\to0^+$ and using right continuity of $g$, we obtain
\begin{equation}
\limsup_{n\to\infty}g(a_n)\le g(L).
\end{equation}
\end{proof}

\section{Main Results}

We now state and prove our main fixed point theorem. The key idea is to use the local uniform convergence of $\varphi_n$ to $\varphi$ on the bounded orbit, together with the Boyd–Wong type estimate, to show that the tail diameters of the orbit shrink to zero. This yields convergence of the iterates to a unique fixed point.

\begin{theorem}[Main Theorem]\label{thm:main}
Let $(X,d)$ be a complete metric space and $T:X\to X$ be continuous. 
Assume that $T$ is an asymptotic pointwise contraction in the sense of Definition \ref{def:main}. 
If there exists $x_0\in X$ such that the orbit $\mathcal{O}(x_0)$ is bounded, then $\{T^n x_0\}$ converges to the unique fixed point of $T$.
\end{theorem}

\begin{proof}
Let $x_0\in X$ have a bounded orbit. 
Set $x_n=T^n x_0$ and $B_0=\mathcal{O}(x_0)=\{x_n:n\ge0\}$. 
Then $B_0$ is bounded, so $\varphi_n\to\varphi$ uniformly on $B_0\times B_0$.

Define the tail diameters
\begin{equation}
b_n = \sup_{i,j\ge n} d(x_i,x_j),\qquad n\in\mathbb{N}.
\end{equation}
Clearly $b_0<\infty$, the sequence $\{b_n\}$ is nonincreasing, and $b_n\to0$ if and only if $\{x_n\}$ is Cauchy.

\paragraph{Step 1: A contraction inequality for $b_n$.}
Given $\varepsilon>0$, by uniform convergence on $B_0\times B_0$ there exists $M\in\mathbb{N}$ such that for all $m\ge M$ and all $u,v\in B_0$,
\begin{equation}
d(T^m u,T^m v) \le \varphi(u,v)+\varepsilon \le \psi(M(u,v))+\varepsilon,
\end{equation}
where the last inequality uses condition (3) of Definition \ref{def:main}.

Fix $n\ge0$ and take arbitrary $i,j\ge n+M$. Write $i=i'+M$, $j=j'+M$ with $i',j'\ge n$. Then
\begin{equation}
d(x_i,x_j)=d(T^M x_{i'},T^M x_{j'}) \le \psi\bigl(M(x_{i'},x_{j'})\bigr)+\varepsilon.
\end{equation}
Now estimate $M(x_{i'},x_{j'})$. By definition,
\[
M(x_{i'},x_{j'})=\max\bigl\{d(x_{i'},x_{j'}),\; d(x_{i'+1},x_{i'}),\; d(x_{j'+1},x_{j'}),\; d(x_{j'},x_{i'+1}),\; d(x_{i'},x_{j'+1})\bigr\}.
\]
All indices appearing are at least $n$ (since $i',j'\ge n$ and adding $1$ does not go below $n$). Therefore each of these distances is bounded by $b_n$. Consequently,
\begin{equation}
M(x_{i'},x_{j'}) \le b_n.
\end{equation}
Because $\psi\le g$ and $g$ is nondecreasing, we obtain
\begin{equation}
d(x_i,x_j) \le g(b_n)+\varepsilon.
\end{equation}
Taking supremum over $i,j\ge n+M$ yields
\begin{equation}
b_{n+M} \le g(b_n)+\varepsilon. \tag{1}
\end{equation}

\paragraph{Step 2: $b_n\to0$.}
Since $\{b_n\}$ is nonincreasing and bounded below, it converges to some limit $L\ge0$. In fact $L=\limsup_{n\to\infty}b_n$. From (1) we have, for the fixed $M$ and $\varepsilon$,
\begin{equation}
b_{n+M} \le g(b_n)+\varepsilon.
\end{equation}
Take $\limsup_{n\to\infty}$ on both sides. Using Proposition \ref{prop:limsup-g} (note that $g$ is nondecreasing and right continuous) we get
\begin{equation}
L \le g(L)+\varepsilon.
\end{equation}
Letting $\varepsilon\to0^+$ gives $L\le g(L)$. If $L>0$, then Lemma \ref{lem:g-properties} implies $g(L)<L$, a contradiction. Hence $L=0$, so $b_n\to0$. Thus $\{x_n\}$ is a Cauchy sequence.

\paragraph{Step 3: Existence and uniqueness of the fixed point.}
Since $X$ is complete, $x_n\to z$ for some $z\in X$. By continuity of $T$,
\begin{equation}
Tz = \lim_{n\to\infty} Tx_n = \lim_{n\to\infty} x_{n+1} = z.
\end{equation}
Hence $z$ is a fixed point.

For uniqueness, let $z_1,z_2$ be fixed points of $T$. Then
\begin{equation}
d(z_1,z_2)=d(T^n z_1,T^n z_2)\le\varphi_n(z_1,z_2).
\end{equation}
Letting $n\to\infty$ and using pointwise convergence,
\begin{equation}
d(z_1,z_2)\le\varphi(z_1,z_2)\le\psi\bigl(M(z_1,z_2)\bigr).
\end{equation}
Because $Tz_i=z_i$, we have $d(Tz_i,z_i)=0$, and also $d(Tz_2,z_1)=d(z_2,z_1)$, $d(Tz_1,z_2)=d(z_1,z_2)$. Thus
\begin{equation}
M(z_1,z_2)=\max\{d(z_1,z_2),0,0,d(z_1,z_2),d(z_1,z_2)\}=d(z_1,z_2).
\end{equation}
Hence
\begin{equation}
d(z_1,z_2)\le\psi\bigl(d(z_1,z_2)\bigr).
\end{equation}
If $d(z_1,z_2)>0$, then $\psi(t)<t$, a contradiction. Therefore $z_1=z_2$.

\paragraph{Step 4: Global attractivity.}
Every point with a bounded orbit satisfies the same assumptions, so its iterates converge to $z$.
This completes the proof.
\end{proof}

\begin{theorem}\label{thm:bounded-set}
Let $M$ be a nonempty bounded closed subset of a complete metric space $(X,d)$. 
If $T:M\to M$ is continuous and an asymptotic pointwise contraction in the sense of Definition \ref{def:main}, then for every $x\in M$, $\{T^n x\}$ converges to the unique fixed point of $T$.
\end{theorem}

\begin{proof}
Every orbit of a point in $M$ is bounded. Theorem \ref{thm:main} applies directly.
\end{proof}

\begin{theorem}\label{thm:unique-attractor}
Let $(X,d)$ be a complete metric space and $T:X\to X$ be continuous and an asymptotic pointwise contraction. Suppose that $T$ has a fixed point $p\in X$ (i.e., $Tp=p$). Then for every $x\in X$, the iterates $T^n x$ converge to $p$.
\end{theorem}

\begin{proof}
First we show that the orbit of any $x\in X$ is bounded. Since $p$ is a fixed point, using condition (1) of Definition \ref{def:main} with $y=p$ gives
\[
d(T^n x, p) = d(T^n x, T^n p) \le \varphi_n(x,p).
\]
The sequence $\{\varphi_n(x,p)\}$ converges pointwise to $\varphi(x,p)$ as $n\to\infty$, hence it is bounded. Thus $\{d(T^n x, p)\}$ is bounded, and consequently $\mathcal{O}(x)$ is bounded.

Now fix an arbitrary $x\in X$. Since $\mathcal{O}(x)$ is bounded, Theorem \ref{thm:main} applies and guarantees that $\{T^n x\}$ converges to a fixed point of $T$. By the uniqueness of fixed points (proved in Theorem \ref{thm:main}), that limit must be $p$. Hence $T^n x \to p$ for all $x\in X$.
\end{proof}

\begin{example}\label{ex:nontrivial}
Let \(X=[0,+\infty)\) be equipped with the usual metric \(d(x,y)=|x-y|\). Define
\[
T(x)=\frac{1}{3}x,\qquad x\ge0.
\]
For each \(n\in\mathbb{N}\) set
\[
\varphi_n(x,y)=\Bigl(\frac13+\frac1{n+1}\Bigr)|x-y|,\qquad \varphi(x,y)=\frac13|x-y|.
\]
Then:
\begin{enumerate}
    \item \(d(T^n x,T^n y)=\dfrac{1}{3^n}|x-y|\le \varphi_n(x,y)\) for all \(x,y\ge0\) because \(\frac1{3^n}\le\frac13+\frac1{n+1}\).
    \item \(\varphi_n\to\varphi\) pointwise on \(X\times X\).
    \item For any bounded set $B\subset X$, say $B\subset[0,R]$, we have
    \[
    \sup_{x,y\in B}|\varphi_n(x,y)-\varphi(x,y)|\le \frac{R}{n+1}\to0,
    \]
    so the convergence is uniform on every bounded set.
    \item The convergence is \emph{not} uniform on the whole \(X\times X\) because the supremum over unbounded pairs is infinite.
    \item \(\varphi(x,y)=\frac13 d(x,y)\le\psi\bigl(M(x,y)\bigr)\) with \(\psi(t)=\frac13 t\). The function \(\psi\) satisfies the classical Boyd--Wong condition: it is nondecreasing, right upper semicontinuous, and \(\psi(t)<t\) for all \(t>0\). Moreover, \(M(x,y)\ge d(x,y)\) holds trivially.
\end{enumerate}
Thus $T$ is an asymptotic pointwise contraction according to Definition \ref{def:main}, but it does \emph{not} satisfy the global uniform convergence required in Kirk's original theorem. The unique fixed point is $0$, and $T^n x\to0$ for every $x\ge0$.
\end{example}

\begin{remark}
Local uniform convergence is strictly weaker than global uniform convergence. 
Example \ref{ex:nontrivial} provides an explicit mapping where \(\varphi_n\to\varphi\) pointwise and uniformly on every bounded set, yet not uniformly on \(X\times X\). 
In such a situation Kirk's original asymptotic contraction theorem does not apply directly, whereas our Theorem \ref{thm:main} still guarantees convergence of iterates to the unique fixed point.
\end{remark}

\section{Declarations}
\begin{itemize}
\item Funding\\ 
 Not applicable
\item Use of Generative-AI tools declaration\\
I declare that no Generative-AI tools were used in the preparation of this manuscript.
\item Conflict of interest/Competing interests\\ 
We recognize no conflicts of interest in the submission
\item Ethics approval\\ 
 Not applicable
\item Consent to participate\\ 
Not applicable
\item Consent for publication\\ 
Not applicable
\item Availability of data and materials\\ 
Not applicable
\item Code availability \\ 
Not applicable
\item Authors' contributions\\ 
The author contributed to this work alone
\end{itemize}

\end{document}